\documentclass{amsart}
\usepackage{amsfonts}
\usepackage{amsmath,amssymb}
\usepackage{amsthm}
\usepackage{amscd}
\usepackage{graphics}
\usepackage{graphicx}
{
\theoremstyle{remark}
\newtheorem{Def}{{\rm Definition}}
\newtheorem{Ex}{{\rm Example}}

}
\newtheorem{Cor}{Corollary}
\newtheorem{Fact}{Fact}
\newtheorem{Prop}{Proposition}
\newtheorem{Thm}{Theorem}
\newtheorem{Lem}{Lemma}

\newtheorem*{Prob*}{Problem}
\begin{document}
\title[Torsion subgroups of homology groups of Reeb spaces of fold maps]{Explicit remarks on the torsion subgroups of homology groups of Reeb spaces of explicit fold maps}
\author{Naoki Kitazawa}
\address{Institute of Mathematics for Industry, Kyushu University, 744 Motooka, Nishi-ku Fukuoka 819-0395, Japan}
\email{n-kitazawa@imi.kyushu-u.ac.jp}
\maketitle
\begin{abstract}
{\it Fold} maps are higher dimensional versions of Morse functions and fundamental and important tools in studying algebraic and differential topological properties of manifolds: as the theory established by Morse and the higher dimensional version, started by Thom and Whitney, later actively studied by Eliashberg, Levine etc. and recently studied by Kobayashi, Saeki, Sakuma etc., explicitly show this.

One of fundamental, important and difficult studies on this field is, constructing explicit fold maps and investigating their source manifolds. As fundamental and strong tools for systematic construction, the author has introduced surgery operations ({\it bubbling operations}) to fold maps, motivated by studies of Kobayashi etc. since 1990. The author has explicitly shown that homology groups of {\it Reeb spaces} of maps constructed by iterations of these operations are generally flexible and restricted in several specific cases. The {\it Reeb space} of a map is defined as the space of all connected components of inverse images of the maps, inheriting fundamental invariants of manifolds such as homology groups etc. in considerable cases. In general, Reeb spaces are fundamental and important tools in the field.

In this paper, we explicitly remark on the torsion subgroups of the homology groups. More precisely, under explicit algebraic constraints, we see explicit strong restrictions related with the torsion subgroups, where the homology groups seem to be very flexible in general. We note that this work is similar to several works by the author before but a work in a new situation and that new technique such as well-known fundamental theory of abstract algebra will be used.      
          

\end{abstract}

\section{Introduction}
First, in this paper, manifolds, maps between them, bundles whose fibers are manifolds etc. are fundamental objects and they are smooth and of class $C^{\infty}$ unless otherwise stated.

We  explain several terminologies. 

For a smooth map, a {\it singular} point is a point such that the rank of the differential at the point drops, the {\it singular set} of the map is the set of all the singular points, a {\it singular value} of the map is the value at a singular point of the map, the {\it singular value set} of the map is the image of the singular set, a {\it regular value} of the map is the point in a complement of the singular value set and  the {\it regular value set} of the map is the complement or the set of all the regular values. 

We call a bundle whose fiber is a topological space $X$ an {\it $X$-bundle}.

We note about homotopy spheres. A homotopy sphere is called an {\it exotic} sphere if it is not diffeomorphic to a standard sphere and an {\it almost-sphere} if it is obtained by gluing two copies of a standard closed sphere on the boundaries by a diffeomorphism. Every homotopy sphere except $4$-dimensional exotic spheres, being undiscovered, is an almost-sphere.
A {\it linear} bundle whose fiber is a standard closed disc (an unit disc) or a standard sphere (an unit sphere) means a bundle whose structure group acts on the fiber linearly.

We also note that most of the content of the introduction is similar to that of \cite{kitazawa6}. 
\subsection{Fold maps and special generic maps}
{\it Fold} maps are higher dimensional versions of Morse functions and fundamental and important tools in the theory of Morse functions and higher dimensional versions and application to algebraic and differential topological theory of manifolds.
\begin{Def}
\label{def:1}
Let $m$ and $n$ positive integers satisfying the relation $m \geq n \geq 1$. A smooth map between $m$-dimensional smooth manifold without boundary and $n$-dimensional smooth manifold without boundary is said to be a {\it fold map} if at each singular point $p$, the form is
$$(x_1, \cdots, x_m) \mapsto (x_1,\cdots,x_{n-1},\sum_{k=n}^{m-i}{x_k}^2-\sum_{k=m-i+1}^{m}{x_k}^2)$$
 for an integer $0 \leq i(p) \leq \frac{m-n+1}{2}$.
\end{Def}

\begin{Prop}
 For any singular point $p$ of the fold map in Definition \ref{def:1}, the $i(p)$ is unique {\rm (}$i(p)$ is called the {\rm index} of $p${\rm )},  the set of all singular points of a fixed index of the map is a closed smooth submanifold of dimension $n-1$ of the source manifold and the restriction map to the singular set is a smooth immersion of codimension $1$.
\end{Prop}
 A fold map is {\it stable} if and only if the restriction map has only normal crossings as self-crossings.
For fundamental stuffs and properties on fold maps, see \cite{golubitskyguillemin} for example. For differential topological theory and advanced related results, see \cite{saeki} for example.

The following class gives explicit examples of fold maps having simple structures.

\begin{Def}
A fold map is said to be {\it special generic} if any singular point $p$ is of index $0$. 
\end{Def}
For example, Morse functions with just two singular points, characterizing spheres topologically (except $4$-dimensional exotic spheres, being undiscovered), and canonical projections of unit spheres are stable and special generic maps such that the singular value sets are embedded spheres. 

As an advanced study, every homotopy sphere of dimension $m>1$ not being an exotic $4$-dimensional sphere admits a special generic stable map into the plane whose singular value set is an embedded circle. 
A homotopy sphere of dimension $m>3$ admitting a special generic map into ${\mathbb{R}}^n$ ($n=m-3,m-2,m-1$) is diffeomorphic to $S^m$. 

\begin{Fact}[\cite{saeki2}] 
\label{fact:1}
A manifold of dimension $m>1$ admits a special generic map into ${\mathbb{R}}^n$ satisfying the relation $m>n \geq 1$ if and only if $M$ is obtained by gluing the following two compact manifolds on the boundaries by a bundle isomorphism between
 the bundles whose fibers are $S^{m-n}$ defined on the boundaries in canonical ways. 
\begin{enumerate}
\item The total space of a linear $D^{m-n+1}$-bundle over the boundary of a compact $n$-dimensional manifold $P$ immersed into ${\mathbb{R}}^n$. 
\item The total space of a bundle over $P$ whose fiber is $S^{m-n}$.
\end{enumerate}
Thus, the Reeb space of a special generic map into an Euclidean space is a smooth  compact manifold $P$, immersed into the target Euclidean space. FIGURE \ref{fig:1} shows the image and the structure of a special generic map into an Euclidean space. 
\begin{figure}
\label{fig:1}
\includegraphics[width=50mm]{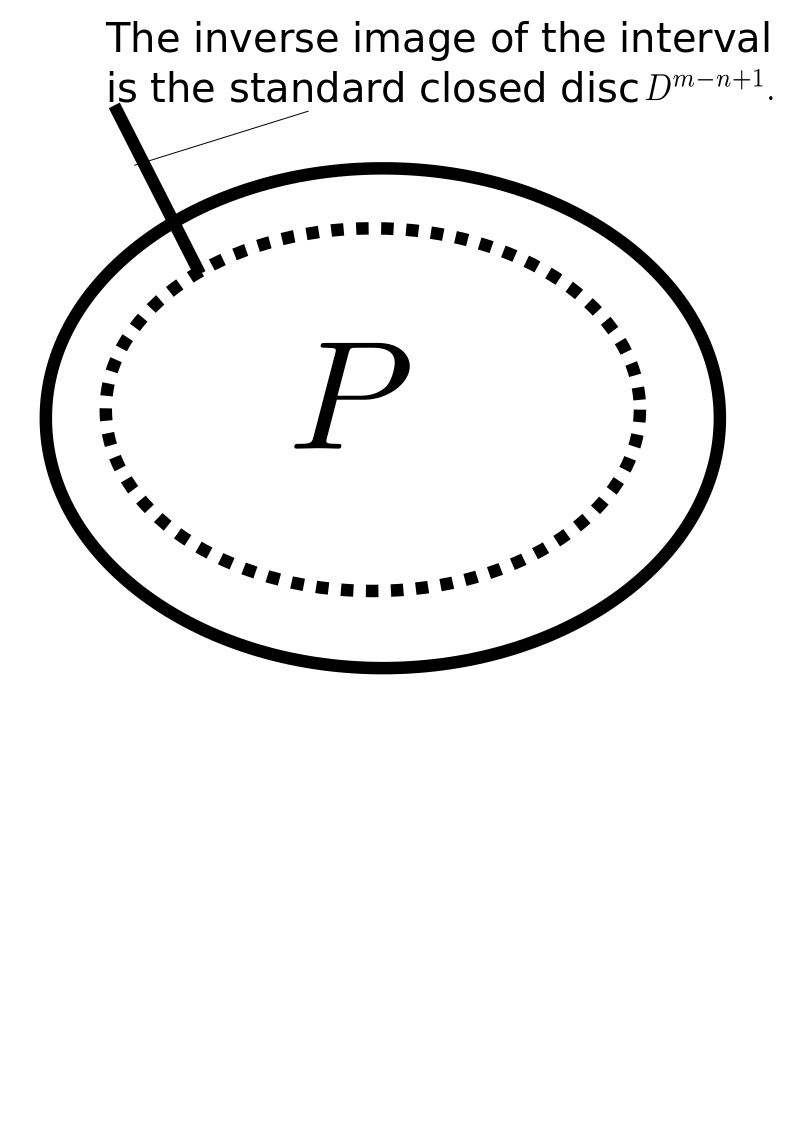}
\caption{The image of a special generic map into ${\mathbb{R}}^n$.}
\end{figure}
\end{Fact}

\subsection{Reeb spaces}
In studying such maps, {\it Reeb spaces} (\cite{reeb}) are fundamental and important. The {\it Reeb} space of a map $c:X \rightarrow Y$ is the space of all connected components of inverse images of $c$ and denoted by $W_c$. We denote the quotient map onto $W_c$ by $q_c:X \rightarrow W_c$ and a uniquely defined map by $\bar{f}$ satisfying the relation $f=\bar{f} \circ q_f$. For example, this is a graph if $c$ is a function whose singular value set is finite, a polyhedron of dimension equal to the dimension of the target manifold if it is a fold map and in considerable cases such as for stable maps, this holds (\cite{shiota}).

\subsection{The existence of fold maps and construction of explicit fold maps}

Stable Morse functions exist densely on any closed manifold. It is also well-known that a closed manifold of dimension larger than $1$ admits a (stable) fold map into the plane if and only if the Euler number is even.
Existence of fold maps into general Euclidean spaces have been studied since Eliashberg published \cite{eliashberg} and \cite{eliashberg2} etc.. Thus the following is a fundamental and important problem.  

\begin{Prob*}
Can we construct explicit fold maps and investigate manifolds admitting them?
\end{Prob*}

Special generic maps are maps having simple differential topological structures and the class of special generic maps give various explicit fold maps. However, the condition is too strong and manifolds admitting such maps are restricted.
Thus fold maps satisfying milder conditions are important. As a simplest class, we introduce {\it round} fold maps introduced by the author in \cite{kitazawa}, \cite{kitazawa2} and \cite{kitazawa3}.
 
We define an equivalence relations on the class of all smooth maps to define {\it round} fold maps.
Two smooth maps $c_1:X_1 \rightarrow Y_1$ and $c_2:X_2 \rightarrow Y_2$ are said to be {\it $C^{\infty}$ equivalent} if
a pair $({\phi}_X,{\phi}_Y)$ of diffeomorphisms exists satisfying the relation ${\phi}_Y \circ c_1=c_2 \circ {\phi}_X$. We also say that $c_1$ is {\it $C^{\infty}$ equivalent to} $c_2$.

\begin{Def}[\cite{kitazawa2}, \cite{kitazawa3}, \cite{kitazawa4} etc.]
\label{def:3}
$f:M \rightarrow {\mathbb{R}}^n$ is said to be a {\it round} fold map if either of the following hold.
\begin{enumerate}
\item $n=1$ holds and $f$ is $C^{\infty}$ equivalent to
 a fold map $f_0:M_0 \rightarrow {\mathbb{R}}$ on a closed manifold $M_0$ such that the following three hold.

\begin{enumerate}
\item $0$ is a regular value of $f_0$.  
\item Two Morse functions defined as ${f_0} {\mid}_{{f_0}^{-1}(-\infty,0]}$ and ${f_0} {\mid}_{{f_0}^{-1}[0,+\infty)}$ are $C^{\infty}$ equivalent.
\item $f_0(S(f_0))$ is the set of all integers whose absolute values are positive and not larger than a positive integer.  
\end{enumerate}
\item $n \geq 2$ holds and $f$ is $C^{\infty}$ equivalent to
 a fold map $f_0:M_0 \rightarrow {\mathbb{R}}^n$ on a closed manifold $M_0$ such that the following three hold.

\begin{enumerate}
\item The singular set $S(f_0)$ is a disjoint union of standard spheres whose dimensions are $n-1$ and consists of $l >0$ connected components.
\item The restriction map $f_0 {\mid}_{S(f_0)}$ is an embedding.
\item Let ${D^n}_r:=\{(x_1,\cdots,x_n) \in {\mathbb{R}}^n \mid {\sum}_{k=1}^{n}{x_k}^2 \leq r \}$. Then, the relation $f_0(S(f_0))={\sqcup}_{k=1}^{l} \partial {D^n}_k$ holds.  
\end{enumerate}
\end{enumerate}
\end{Def}

\begin{Ex}
\label{ex:1}
 (\cite{kitazawa},\cite{kitazawa2},\cite{kitazawa4} etc.)
Let $m$ and $n$ be positive integers satisfying the relation $m>n$.
Let $M$ be an $m$-dimensional closed and connected manifold and $\Sigma$ be an ($m-n$)-dimensional almost-sphere. Then the following are equivalent.
\begin{enumerate}
\item $M$ is the total space of a smooth bundle over $S^n$ whose fiber is $\Sigma$.
\item 
\begin{figure}
\label{fig:2}
\includegraphics[width=40mm]{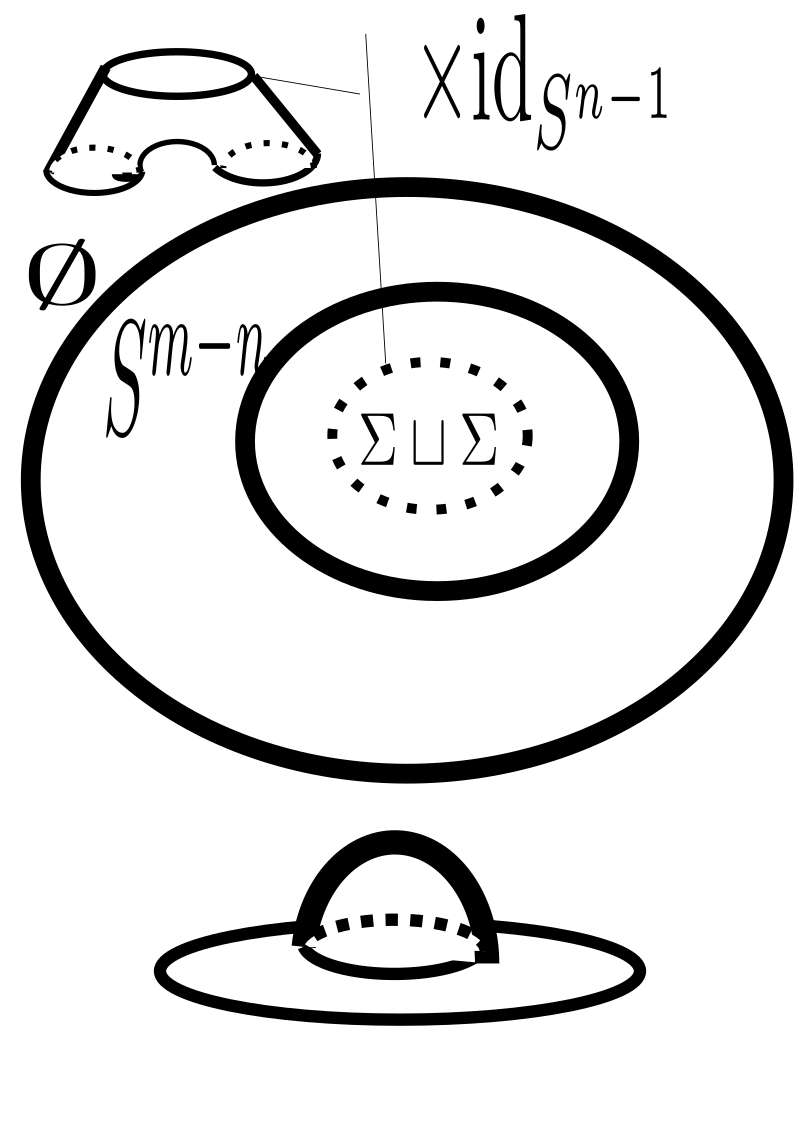}
\caption{The image of a round fold map of Example \ref{ex:1} and the Reeb space. $\emptyset$ and manifolds represent inverse images, thick curves represent the singular value set and the pair of pants represent a Morse function on the manifold obtained by removing the interior of an ($m-n+1$)-dimensional standard closed disc in the interior of the cylinder $\Sigma \times [-1,1]$.}
\end{figure}
Either of the following holds (see FIGURE \ref{fig:2}: only the latter case is presented).
\begin{enumerate}
\item $n=1$ and $M$ admits a round fold map with just four singular points. The inverse images of regular values in the five connected components of the regular value set are $\emptyset$, $S^{m-n}$, $\Sigma \sqcup \Sigma$, $S^{m-n}$ and $\emptyset$, respectively.
\item $n \geq 2$ and $M$ admits a round fold map such that the singular set is the disjoint union of two copies of $S^{n-1}$. The inverse images of regular values in the three connected components of the regular value set are $\emptyset$, $S^{m-n}$ and $\Sigma \sqcup \Sigma$, respectively. Furthermore, on the complement of
the interior of an $n$-dimensional standard closed disc smoothly embedded in the connected component of the center of the regular value set and its inverse image, the map is $C^{\infty}$ equivalent to a product of a Morse function with just two singular points on the cylinder $\Sigma \times [-1,1]$ such that the boundary coincides with the inverse image of the minimum and the identity map ${\rm id}_{S^{n-1}}$.
\end{enumerate}
\end{enumerate}

\end{Ex}

\subsection{Surgery operation to construct explicit fold maps more}

To obtain maps other than the presented maps explicitly and systematically, the author has introduced {\it {\rm (}normal{\rm )} bubbling operations} as surgery operations to stable fold maps. We remove the interior of a small regular neighborhood of a compact bouquet of closed, connected and orientable submanifolds (resp. a small closed tubular neighborhood of a closed and connected and orientable submanifold) in the regular value of an original fold map and a connected component of its inverse image and after that, we attach a new map such that the singular value set is regarded as the boundary of the regular (resp. closed tubular) neighborhood and in the interior of the target space and that the map obtained by the restriction to the singular set is an embedding. In this way we obtain a new fold map.
These operations are generalizations of {\it bubbling surgeries} by Kobayashi \cite{kobayashi3}: Kobayashi`s surgery is the case where the bouquet is a point. 

\begin{Ex}
\label{ex:2}
The map presented in FIGURE \ref{fig:2} is obtained by a bubbling surgery to a stable special generic map whose singular set is a standard sphere.
\end{Ex}

\subsection{Homological properties of maps obtained by surgery operations and main works and the content of this paper}
The author constructed maps by finite iterations of such surgery operations with setting and challenging the following problem.
   
\begin{Prob*}
Let $R$ be a principal ideal domain and for a fold map $f:M \rightarrow N$ from a closed and connected manifold of dimension $m$ into an manifold without boundary of dimension $n$ satisfying the relation $m>n \geq 1$, let $f^{\prime}:{M}^{\prime} \rightarrow N$ be a fold map obtained by a finite iteration of normal bubbling operations to $f$. Then for any integer $0 \leq j \leq n$, we can set a finitely generated module $G_j$ over $R$ so
  that $G_0$ is a trivial $R$-module and that $G_n$ is not
 a trivial $R$-module and the module $H_j(W_{{f}^{\prime}};R)$ is isomorphic to $H_j(W_f;R) \oplus G_j$. 

 Conversely, for a suitable family $\{G_j\}_{j=0}^{n}$ of modules, can we construct
 a suitable map $f^{\prime}$ satisfying the condition by a finite iteration of bubbling operations starting from $f$?
\end{Prob*}
Example \ref{ex:1} or \ref{ex:2} accounts for the case where $G_j$ is zero for $1 \leq j \leq n-1$ and isomorphic to $\mathbb{Z}$ for $j=n$. In \cite{kitazawa5}, we have shown that the modules are flexible by explicit construction. They will be explicitly presented in Propositions \ref{prop:3}, \ref{prop:4} and \ref{prop:5}. We have also seen several restrictions.
In \cite{kitazawa6}, we show more explicit restrictions. More precisely, we have investigated cases where the numbers of non-trivial $G_j$ are small. In addition, in showing such facts, we have found new sufficient conditions to
 obtain $f^{\prime}$ from the original map $f$ which we could not find in \cite{kitazawa5}.

In this paper, as a related study, we investigate effects the torsion subgroups of the groups $G_j$ give to $\{G_j\}$.

 Constructing fold maps having Reeb spaces whose homology groups have non-trivial torsion subgroups, or obtaining manifolds whose homology groups have non-trivial torsion subgroups admitting explicit fold maps seems to be difficult in general. In \cite{nishioka}, Nishioka has studied homology groups of manifolds admitting special generic maps into Euclidean spaces precisely, explicitly shown that torsion subgroups of homology groups of suitable degrees are zero if constraints on suitable parts of homology groups are posed and as application, for simply connected and closed manifolds of dimension $5$, completely classified in  \cite{barden}, has determined whether the manifolds admit special generic maps or not completely.  
In \cite{kitazawa5} and \cite{kitazawa6}, cases where $G_j$ have torsion subgroups have been also studied. For example, several suitable maps have been constructed and homology groups of their Reeb spaces have been studied. 
However, these studies give only several small families of desired maps and the Reeb spaces.   

In this stream, as a general study, general restrictions on Reeb spaces whose homology groups have non-trivial torsion subgroups will be explicitly studied.
We apply methods which have been used and which have not been used in \cite{kitazawa5} (\cite{kitazawa6}) such as several methods in the fundamental theory of commutative groups.   

The content of the paper is as the following. 
In the next section, we introduce and review {\it bubbling operations} based on \cite{kitazawa5}. 
The next section is for main works of the present paper. We consider several explicit situations and investigate maps and Reeb spaces. We present several patterns such that the torsion subgroups of homology groups of Reeb spaces may not be trivial explicitly. Last, we introduce the following as Fact \ref{fact:3}: for a stable fold map such that inverse images of regular values are disjoint unions of almost-spheres satisfying appropriate differential topological conditions, homology groups and homotopy groups of the source manifold and the Reeb space whose degrees are smaller than the difference of the dimensions of the source and the target manifolds are isomorphic. Special generic maps, maps presented in Example \ref{ex:1} etc. satisfy the assumption of this fact. This gives clues to know homology and homotopy groups of the source manifolds from the Reeb spaces. 

\subsection{Acknowledgement}
\thanks{
The author is a member of the project and supported by the project Grant-in-Aid for Scientific Research (S) (17H06128 Principal Investigator: Osamu Saeki)
"Innovative research of geometric topology and singularities of differentiable mappings"

( 
https://kaken.nii.ac.jp/en/grant/KAKENHI-PROJECT-17H06128/
).}

\section{bubbling operations}

\begin{Def}[\cite{kitazawa5}]
\label{def:4}
Let $f:M \rightarrow N$ be a fold map from a closed and connected manifold of dimension $m$ into an manifold without boundary of dimension $n$ satisfying the relation $m>n \geq 1$, let $R$ be a connected component of the regular
 value set $N-f(S(f))$. Let $S$ be a bouquet of a finite number of connected and orientable closed submanifolds (resp. a connected and orientable closed submanifold) of
 $R$ and $N(S)$, ${N(S)}_i$ and ${N(S)}_o$ be small regular neighborhoods
 of $S$ in $R$ such that the relations ${N(S)}_i \subset N(S) \subset {N(S)}_o$, ${N(S)}_i \subset {\rm Int }N(S)$ and $N(S) \subset {\rm Int} {N(S)}_o$ hold. Let
 $f^{-1}({N(S)}_o)$ have a connected component $P$ such that $f {\mid}_{P}$
 makes $P$ a bundle over ${N(S)}_o$.

Furthermore we assume that there exist an $m$-dimensional closed manifold $M^{\prime}$ and
 a fold map $f^{\prime}:M^{\prime} \rightarrow {\mathbb{R}}^n$
 satisfying the following.
\begin{enumerate}
\item $M-{\rm Int} P$ is a compact submanifold (with non-empty boundary) of $M^{\prime}$ of dimension $m$.
\item $f {\mid}_{M-{\rm Int} P}={f}^{\prime} {\mid}_{M-{\rm Int} P}$ holds.
\item ${f}^{\prime}(S({f}^{\prime}))$ is the disjoint union of $f(S(f))$ and $\partial N(S)$.
\item $(M^{\prime}-(M-P)) \bigcap {{f}^{\prime}}^{-1}({N(S)}_i)$ is empty or ${{f}^{\prime}} {\mid}_{(M^{\prime}-(M-P)) \bigcap {{f}^{\prime}}^{-1}({N(S)}_i)}$ makes $(M^{\prime}-(M-P)) \bigcap {f^{\prime}}^{-1}({N(S)}_i)$ a bundle over ${N(S)}_i$.
\end{enumerate}
 These assumptions enable us to consider the procedure of constructing $f^{\prime}$ from $f$ and we
 call it a {\it {\rm (}normal{\rm )} bubbling operation} to $f$ and, ${\bar{f}}^{-1}(S) \bigcap q_f(P)$, which is homeomorphic
 to $S$, the {\it generating polyhedron {\rm (}manifold{\rm )}} of the bubbling operation. 
 
\begin{enumerate}
\item Let us suppose the following additional conditions.
\begin{enumerate}
\item
 ${{f}^{\prime}} {\mid}_{(M^{\prime}-(M-P)) \bigcap {f^{\prime}}^{-1}({N(S)}_i)}$ makes $(M^{\prime}-(M-P)) \bigcap {f^{\prime}}^{-1}({N(S)}_i)$ the disjoint union of two bundles over ${N(S)}_i$, then the procedure is called a {\it normal M-bubbling operation} to $f$. Note that the original inverse image having no singular points is represented as a connected sum of new two manifolds appearing as fibers of the two bundles.
\item ${{f}^{\prime}} {\mid}_{(M^{\prime}-(M-P)) \bigcap {f^{\prime}}^{-1}({N(S)}_i)}$ makes $(M^{\prime}-(M-P)) \bigcap {f^{\prime}}^{-1}({N(S)}_i)$ the disjoint union of two bundles over ${N(S)}_i$ and the fiber of one of the bundles is an almost-sphere, then the procedure is called a {\it normal S-bubbling operation} to $f$. Note that this operation is also a normal M-bubbling operation.
\end{enumerate}
\item As extra assumptions, if the following two hold, then the procedure is called a {\it trivial bubbling operation}. 
\begin{enumerate}
\item The map ${f^{\prime}} {\mid}_{{f^{\prime}}^{-1}({N(S)}_o-{\rm Int} {N(S)}_i)}$ is $C^{\infty}$ equivalent to the product of a Morse function with just one singular point and the identity
 map ${\rm id} {\mid}_{\partial N(S)}$.
\item The map ${f^{\prime}} {\mid}_{{{f}^{\prime}}^{-1}({N(S)}_i)}$ makes $(M^{\prime}-(M-P)) \bigcap {f^{\prime}}^{-1}({N(S)}_i)$ a trivial bundle over ${N(S)}_i$.
\end{enumerate}
\end{enumerate}
\end{Def}
\begin{Fact}[\cite{kitazawa2}, \cite{kitazawa4} etc.]
\label{fact:2}
Let $m$ and $n$ be positive integers satisfying the relation $m \geq 2n$ and let $M$ be an $m$-dimensional closed and connected manifold. Then the following are equivalent.
\begin{enumerate}
\item $M$ is represented as a connected sum of $l>0$ manifolds regarded as the total spaces bundles whose fibers are $S^{m-n}$ over $S^n$.
\item $M$ admits a round fold map obtained by $l$-times trivial normal S-bubbling operations starting from a canonical projection of a unit sphere, more generally, a special generic map into the plane or higher dimensional versions of FIGURE \ref{fig:1} such that inverse images of regular values are disjoint unions of standard spheres.
\end{enumerate}
\end{Fact}
\begin{figure}
\label{fig:3}
\includegraphics[width=40mm]{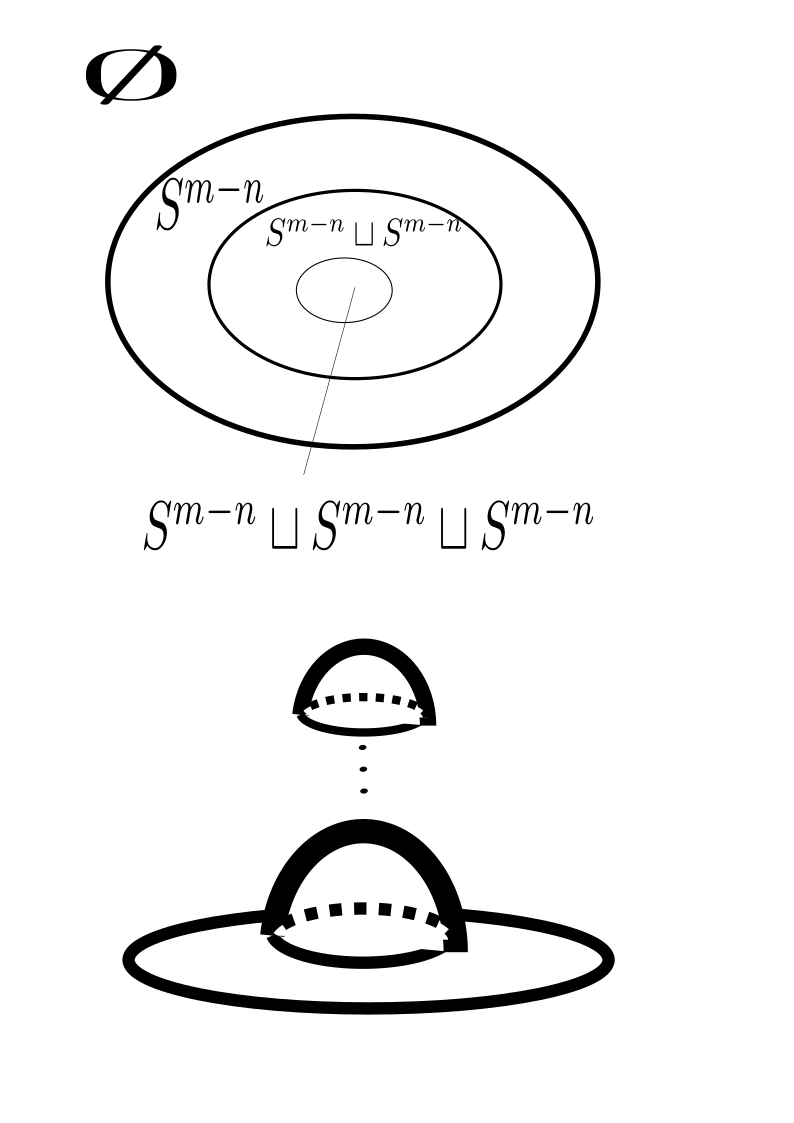}
\caption{The image of a round fold map of Fact \ref{fact:2} and the Reeb space: $\emptyset$ and manifolds represent inverse images, thick curves represent the singular value set and the Reeb space is obtained by piling discs one after another by identifying the boundaries with spheres in the interiors of the discs on the tops starting from the disc in the bottom.}
\label{fig:3}
\end{figure}
See FIGURE \ref{fig:3}. Note that this can be regarded as an extension of Example \ref{ex:1} where the relation $m \geq 2n$ holds with $\Sigma=S^{m-n}$.

In this paper, we only use the following trivial M-bubbling operations essentially as in \cite{kitazawa5}.  
\begin{Ex}
\label{ex:3}
Let $f:M \rightarrow N$ be a fold map from a closed and connected manifold of dimension $m$ into an manifold without boundary of dimension $n$ satisfying the relation $m>n \geq 1$, let $R$ be a connected component of the regular
 value set $N-f(S(f))$. Let $S$ be a connected and orientable closed submanifold of
 $R$ such that there exists a connected component $P$ of $f^{-1}(S)$ and $f {\mid}_{P}$ gives a trivial bundle (take $S$ in an open disc in $R$ for example). Then we can consider a trivial normal M-bubbling operation whose generating manifold is $q_f(P)$ so that the pair of the resulting new two connected components of an inverse image having no singular points or the pair of the fibers of the resulting two bundles explained in Definition \ref{def:1} can be any pair of manifolds by a connected sum of which we obtain the original connected component of an inverse image having no singular points. Note that a trivial S-bubbling operation is a specific case.   
\end{Ex}

We present several results of \cite{kitazawa5} (\cite{kitazawa6}) with sketches of proofs. 
For a finitely generated module $G$ over a principal ideal domain $R$, we denote the rank
 by ${\rm rank}_R G$.

\begin{Prop}[\cite{kitazawa5}]  
\label{prop:2}
Let $f:M \rightarrow N$ be a fold map from a closed and connected manifold of dimension $m$ into an manifold without boundary of dimension $n$ satisfying the relation $m>n \geq 1$, let $R$ be a connected component of the regular
 value set $N-f(S(f))$. Let $f^{\prime}:{M}^{\prime} \rightarrow N$ be a fold map obtained by an
 M-bubbling operation to $f$. Let $S$ be the generating polyhedron of the M-bubbling
 operation. Let $k$ be a positive integer and $S$ be represented as the bouquet of submanifolds $S_j$ where $j$ is an integer satisfying $1 \leq j \leq k$.  
Then, for any principal ideal domain $R$ and any integer $0\leq i<n$, we have

$$H_{i}(W_{{f}^{\prime}};R) \cong H_{i}(W_f;R) \oplus {\oplus}_{j=1}^{k} (H_{i-(n-{\dim} S_j)}(S_j;R))$$

and we also have $H_{n}(W_{{f}^{\prime}};R) \cong H_{n}(W_f;R) \oplus R$.
\end{Prop}

\begin{proof}[Sketch of the proof]
For each $S_j$, we can take a small closed tubular neighborhood, regarded as
 the total space of a linear $D^{n-\dim S_j}$-bundle over $S_j$. By the definition of an M-bubbling operation, we can
 see that a small regular neighborhood of $S$ is represented as a boundary connected sum of the
 closed tubular neighborhoods. $W_{f^{\prime}}$ is regarded as a polyhedron obtained by attaching a manifold represented as a connected sum of total spaces
 of linear $S^{n-\dim S_j}$-bundles over $S_j$ ($1 \leq j \leq k$) by considering $D^{n-\dim S_j}$ in the beginning of
 this proof as a hemisphere of $S^{n-\dim S_j}$ and identifying the subspace obtained by restricting the
 space to fibers $D^{m-\dim S_j}$ with the original regular neighborhood. Note that the connected sum of total spaces of linear $S^{n-\dim S_j}$-bundles
 over $S_j$ ($1 \leq j \leq k$) is regard as a connected sum of product bundles in investigating the homology groups. In fact, the bundles admit sections, corresponding
 to the submanifolds $S_j$ and regarded as sections whose values are the points corresponding to the origin
 in the fiber $D^{n-\dim S_j} \subset S^{n-\dim S_j}$. 

From the observation on the relation of the topologies of $W_f$ and $W_{f^{\prime}}$, we have the result.
\end{proof}

\begin{Cor}[\cite{kitazawa5}]
\label{cor:1}
In Problem in the introduction, $G_{n-1}$ must be free.
\end{Cor}
\begin{proof}
Each $S_j$ is closed, connected and orientable and we can apply the Poincar\'e duality theorem. 
\end{proof}
\begin{Prop}[\cite{kitazawa5}]  
\label{prop:3}
Let $R$ be a principal ideal domain. 
For any integer $0 \leq j \leq n$, we define $G_j$ as a free finitely generated $R$-module so
  that $G_0$ is a trivial $R$-module and that $G_n$ is not
 a trivial $R$-module. Then, by a finite iteration
 of M-bubbling {\rm (}S-bubbling{\rm )} operations starting from $f$, we obtain a fold map $f^{\prime}$ and $H_j(W_{{f}^{\prime}};R)$ is isomorphic to $H_j(W_f;R) \oplus G_j$. 
\end{Prop}
\begin{proof}[Sketch of the proof]
We can prove the result by an explicit finite iteration of S-bubbling operations with Proposition \ref{prop:2}. We can choose a family of disjoint bouquets of standard spheres and points in $W_f-q_f(S(f))$ satisfying the following.
\begin{enumerate}
\item Each bouquet is in an open disc in $W_f-q_f(S(f))$ and the image by the map $\bar{f}$ is in the regular value set of $f$.
\item The family includes just ${\rm rank}_R G_n$ bouquets.
\item For the family $\{S_k\}_{k=1}^{{\rm rank}_R G_n}$ of the just ${\rm rank}_R G_n$ bouquets, $S_k$ is a bouquet consisting of just $l_{k,j}$ ($n-j$)-dimensional spheres for $1 \leq j \leq n-1$ or just a point. 
\item ${\Sigma}_{k=1}^{{\rm rank}_R G_n} l_{k,j}={\rm rank}_R G_j$ 
\end{enumerate}
For the family of the bouquets, we can perform trivial S-bubbling operations whose generating polyhedra are the chosen bouquets one after another. Thus we have a desired map.
\end{proof}
\begin{Prop}[\cite{kitazawa5}]  
\label{prop:4}
Let $R$ be a principal ideal domain. 
For any integer $0 \leq j \leq n$, we define $G_j$ as a free finitely generated $R$-module so
  that $G_0$ is a trivial $R$-module, that $G_n$ is not
 a trivial $R$-module and that the relation ${\Sigma}_{k=1}^{n-1}{{\rm rank}_R G_k} \leq {\rm rank}_R G_n$ holds. Then, by a finite iteration
 of normal M-bubbling {\rm (}S-bubbling{\rm )} operations starting from $f$, we obtain a fold map $f^{\prime}$ and $H_j(W_{{f}^{\prime}};R)$ is isomorphic to $H_j(W_f;R) \oplus G_j$. 
\end{Prop}
\begin{proof}
We can choose a family of standard spheres and points in $W_f-q_f(S(f))$ satisfying the following.
\begin{enumerate}
\item The family includes just ${\rm rank}_R G_j$ copies of $S^{n-j}$ for $1 \leq j \leq n-1$.
\item The family includes just ${\rm rank}_R G_n- {\Sigma}_{k=1}^{n-1}{{\rm rank}_R G_k}$.
\end{enumerate}
For the family of the spheres and the points, we can perform trivial normal M-bubbling (S-bubbling) operations whose generating manifolds are the chosen spheres or points one after another. Thus we have a desired map.
\end{proof}
Example \ref{ex:1} and Fact \ref{fact:2} account for the case $G_j=\{0\}$ ($0 \leq j \leq n-1$) of Propositions \ref{prop:3} and \ref{prop:4} explicitly.

\begin{Prop}[\cite{kitazawa5}] 
\label{prop:5}
Let $f:M \rightarrow N$ be a fold map from a closed and connected manifold of dimension $m$ into a manifold without boundary of dimension $n$ satisfying the relation $m>n \geq 1$.
Let $R$ be a principal ideal domain. For any integer $0 \leq j \leq n$ and a finitely generated module $G_j$ over $R$ satisfying that $G_0$ is trivial and that $G_n$ is not zero. Assume that by
 a finite iteration of M-bubbling operations whose generating polyhedra are in open discs of $W_f$ to a map $f:M \rightarrow N$, we can obtain a fold map $f^{\prime}$ such
  that the module $H_j(W_{{f}^{\prime}};R)$ is isomorphic to the module $H_j(W_f;R) \oplus G_j$.
Then $G_n$ is free from Proposition \ref{prop:2} and let $H$ be a non-trivial submodule of $G_n$, which is free.

Then we can obtain a fold map $f^{\prime \prime}$ such
  that the module $H_j(W_{{f}^{\prime \prime}};R)$ is isomorphic to the module $H_j(W_f;R) \oplus G_j$ for $0 \leq j \leq n-1$ and that the module $H_n(W_{{f}^{\prime \prime}};R)$ is isomorphic to the module $H_n(W_f;R) \oplus H$ by a finite iteration of M-bubbling {\rm (}S-bubbling{\rm )} operations whose generating polyhedra are in open discs of $W_f$ to the original map.
\end{Prop}
We omit the proof of Proposition \ref{prop:5}

 \section{Main results}
Observing Propositions \ref{prop:3}, \ref{prop:4} and \ref{prop:5}, the family $\{G_j\}$ of groups realized by constructing a new fold map $f^{\prime}$ seems to be flexible. On the other hand, in this paper, we will discover several explicit restrictions on the groups under the existence of several groups of the family whose torsion subgroups are not trivial. 

We review a fundamental fact on the structure of a finitely generated commutative group.
We denote the set of all positive integers $q$ satisfying the relation $q \leq p>0$ by $\lfloor p \rfloor$. 
We denote the finite cyclic group of order $k>0$ by ${\mathbb{Z}}_k$.
The following fact is fundamental. It is implicitly used in the previous sections.
\begin{Fact}
\label{fact:3}
Let $G$ be a non-trivial finitely generated commutative group.
Then there exist a non-negative integer $f$, a finite set $A$ of prime numbers and a map
 $n:A \rightarrow \mathbb{N}$ and for each $p \in A$, a map $m(p):\lfloor n(p) \rfloor \rightarrow \mathbb{N} \sqcup \{0\}$ satisfying $m(p)(n(p))>0$ {\rm (}we denote the map on $A$ corresponding the previous map $m(p)$ for each $p \in A$ by $m${\rm )} and $G$ is isomorphic to the direct sum $${\mathbb{Z}}^f \oplus {\oplus}_{p \in A} ({\oplus}_{j=1}^{n(p)} ({\oplus}_{k=1}^{m(p)(j)} {{\mathbb{Z}}_{p^j}}))$$
of cyclic groups. Moreover, we can take the tuple $(f,A,n,m)$ uniquely: for example, $f$ is the rank of $G$ or ${\rm rank}_{\mathbb{Z}} G$.     
\end{Fact}
We denote the summand represented by ${\oplus}_{p \in A} ({\oplus}_{j=1}^{n(p)} ({\oplus}_{k=1}^{m(p)(j)} {{\mathbb{Z}}_{p^j}}))$ by $TG$ and call this the {\it torsion subgroup} of $G$.
\subsection{A case where there exists just one $G_j$ whose torsion-group is non-trivial in Problem in the introduction}
\begin{Thm}
\label{thm:1}
In Problem in the introduction, if there exists just one group $G_j$ in the family of the groups such that the torsion subgroup $T_{G_j}$ is not trivial and we can construct a map $f^{\prime}$ satisfying the condition, then the relations $j<n-1$ and $2j-n+1>0$ hold and ${\rm rank}_{\mathbb{Z}} G_{2j-n+1}$ is positive.  
\end{Thm}
\begin{proof}
From Corollary \ref{cor:1}, we have the relation $j<n-1$.
From Proposition \ref{prop:2} together with the Poincar\'e duality theory and the universal coefficient theorem, there exists a generating polyhedron represented as a bouquet of a finite number of closed, connected and orientable manifolds and the family of manifolds includes a ($2n-2j-1$)-dimensional manifold $S_i$ such that the group $H_{n-j-1}(S_i;\mathbb{Z})$ is isomorphic to a group whose torsion subgroup is a non-trivial subgroup of $T_{G_j} \subset G_j$ and the group $H_{j^{\prime}}(S_i;\mathbb{Z})$ is trivial for $0<j^{\prime}<2n-2j-1$, $j^{\prime} \neq n-j-1$. We can see the inequality $n-(2n-2j-1)=2j-n+1>0$ holds. The last part follows immediately from Proposition \ref{prop:2}.
\end{proof}
\begin{Thm}
\label{thm:2}
In Problem in the introduction, if there exists just one group $G_j$ in the family of the groups whose torsion subgroup is not trivial and there exists a closed, connected and orientable {\rm (}$2n-2j-1${\rm )}-dimensional manifold $S^{\prime}$ we can embed into ${\mathbb{R}}^n$ such that the group $H_{n-j-1}(S^{\prime};\mathbb{Z})$ is isomorphic to $T G_j$ and that the group $H_{j^{\prime}}(S^{\prime};\mathbb{Z})$ is trivial for $0<j^{\prime}<2n-2j-1$, $j^{\prime} \neq n-j-1$, then we can construct a map $f^{\prime}$ satisfying the condition by a finite iteration of normal M-bubbling operations starting from a fold map $f$.  
\end{Thm}
\begin{proof}
This can be shown by applying Proposition \ref{prop:2}.
We take a bouquet of a finite number of closed, connected and orientable submanifolds in an open disc in $W_f-q_f(S(f))$ and ${\rm rank}_{\mathbb{Z}} G_n-1$ points. For the bouquet, we take the following manifolds.
\begin{enumerate}
\item $S^{\prime}$.
\item ${\rm rank}_{\mathbb{Z}} G_{j^{\prime}}$ copies of the ($n-j^{\prime}$)-dimensional standard sphere $S^{n-j^{\prime}}$ ($1 \leq j^{\prime} <2j-n+1, 2j-n+1 < j^{\prime} \leq n-1$).
\item ${\rm rank}_{\mathbb{Z}} G_{j^{\prime}}-1$ copies of the ($n-j^{\prime}$)-dimensional standard sphere $S^{n-j^{\prime}}$ ($j^{\prime}=2j-n+1$).
\end{enumerate}
We perform trivial M-bubbling operations whose generating polyhedra are the bouquet or the points one after another.
\end{proof}
For explicit manifolds for $S^{\prime}$, we use $3$-dimensional closed, connected and orientable manifolds, which we can embed into ${\mathbb{R}}^5$ as shown in \cite{wall} etc., $5$-dimensional closed and simply-connected manifolds such that the torsion subgroups of the 2nd homology groups whose coefficients are $\mathbb{Z}$ are non-trivial (\cite{barden}), $7$-dimensional closed and $2$-connected manifolds such that the torsion subgroups of the 3rd homology groups whose coefficients are $\mathbb{Z}$ are non-trivial (\cite{crowleyescher}: $S^3$-bundles over $S^4$, investigated there, have been studied also in \cite{milnor} etc.), etc..
\subsection{A case where there exist just two $G_j$ whose torsion-groups are non-trivial in Problem in the introduction}
\begin{Thm}
\label{thm:3}
In Problem in the introduction, if there exist just two groups $G_{j_1}$ and $G_{j_2}$ {\rm (}$0<j_1<j_2<n${\rm )} in the family of the groups whose torsion subgroups are not trivial and we can construct a map $f^{\prime}$ satisfying the condition, then the relation $j_1,j_2<n-1$ holds and the following hold.
\begin{enumerate}
\item Let $G_{j_1}$ be not represented as a summand of $G_{j_2}$. Then the relations $2n-2j_1-1<n$ and ${\rm rank}_{\mathbb{Z}} G_{2j_1-n+1}>0$ hold.
\item Conversely, let $G_{j_2}$ be not represented as a summand of $G_{j_1}$. Then the relations $2n-2j_2-1<n$ and ${\rm rank}_{\mathbb{Z}} G_{2j_2-n+1}>0$ hold.
\end{enumerate}
In both cases, the number of $j$ satisfying ${\rm rank}_{\mathbb{Z}} G_{j}>0$ hold must be at least $2$.
\end{Thm}
\begin{proof}
From Corollary \ref{cor:1}, we have the relation $j_1,j_2<n-1$. We show the former case of the numbered cases. By virtue of Proposition \ref{prop:2} and the Poincar\'e duality theorem together with the universal coefficient theorem, for a bouquet of a finite number of closed, connected and orientable manifolds, the family of manifolds includes a ($2n-2j_1-1$)-dimensional manifold $S_i$ such that the group $H_{n-j_1-1}(S_i;\mathbb{Z})$ is isomorphic to a group whose torsion subgroup is a non-trivial subgroup of $G_{j_1}$ and that the group $H_{j^{\prime}}(S_i;\mathbb{Z})$ is trivial for $0<j^{\prime}<2n-2j_1-1$, $j^{\prime} \neq n-j_1-1$. We can show the latter case similarly. The last statement follows easily from the discussion with Proposition \ref{prop:2}.
\end{proof}
\begin{Thm}
\label{thm:4}
In Problem in the introduction, let there exist just two groups $G_{j_1}$ and $G_{j_2}$ {\rm (}$j_1<j_2${\rm )} in the family of the groups whose torsion subgroups are not trivial.
\begin{enumerate}
\item Assume that there exist {\rm (}$2n-2j_1-1${\rm )}-dimensional and {\rm (}$2n-2j_2-1${\rm )}-dimensional closed, connected and orientable manifolds $S_1$ and $S_2$ we can embed into ${\mathbb{R}}^n$ such that the groups $H_{n-j_1-1}(S_1;\mathbb{Z})$ and $H_{n-j_2-1}(S_2;\mathbb{Z})$ are isomorphic to $G_{j_1}$ and $G_{j_2}$, respectively, and that the group $H_{j}(S_i;\mathbb{Z})$ is trivial for $0<j<2n-2j_i-1$, $j \neq n-j_i-1$ for $i=1,2$.
\begin{enumerate}
\item Let $G_{j_1}$ be not represented as a summand of $G_{j_2}$. Let ${\rm rank}_{\mathbb{Z}} G_{2j_i+1-n}>0$ hold for $i=1,2$. Then we can construct a map $f^{\prime}$ satisfying the condition.
\item Conversely, let $G_{j_2}$ be not represented as a summand of $G_{j_1}$. Let ${\rm rank}_{\mathbb{Z}} G_{2j_i+1-n}>0$ hold for $i=1,2$. Then we can construct a map $f^{\prime}$ satisfying the condition.
\end{enumerate}
\item Assume that $G_{j_2}$ is represented as the direct sum of $G_{j_1}$ and a group $G$. Assume that there exist {\rm (}$2n-2j_2-1${\rm )}-dimensional and {\rm (}$2n-j_1-j_2-1${\rm )}-dimensional closed, connected and orientable manifolds $S_1$ and $S_2$ we can embed into ${\mathbb{R}}^n$ satisfying the following.
\begin{enumerate}
\item The groups $H_{n-j_2-1}(S_1;\mathbb{Z})$ and $H_{n-j_2-1}(S_2;\mathbb{Z})$ are isomorphic to $G$ and $G_{j_1}$, respectively.
\item The group $H_{n-j_1-1}(S_2;\mathbb{Z})$ is isomorphic to $G_{j_1}$.
\item The group $H_{j}(S_1;\mathbb{Z})$ is trivial for $0<j<2n-2j_2-1$, $j \neq n-j_2-1$.
\item The group $H_{j}(S_2;\mathbb{Z})$ is trivial for $0<j<2n-j_1-j_2-1$, $j \neq n-j_1-1,n-j_2-1$. 
\end{enumerate} 
 Let ${\rm rank}_{\mathbb{Z}} G_{2j_2+1-n}>0$ and ${\rm rank}_{\mathbb{Z}} G_{j_1+j_2+1-n}>0$. Then we can construct a map $f^{\prime}$ satisfying the condition.
\item Assume that $G_{j_1}$ is represented as the direct sum of $G_{j_2}$ and a group $G$. Assume that there exist {\rm (}$2n-2j_1-1${\rm )}-dimensional and {\rm (}$2n-j_1-j_2-1${\rm )}-dimensional closed, connected and orientable manifolds $S_1$ and $S_2$ we can embed into ${\mathbb{R}}^n$ satisfying the following.
\begin{enumerate}
\item The groups $H_{n-j_1-1}(S_1;\mathbb{Z})$ and $H_{n-j_1-1}(S_2;\mathbb{Z})$ are isomorphic to $G$ and $G_{j_2}$, respectively.
\item The group $H_{n-j_2-1}(S_2;\mathbb{Z})$ is isomorphic to $G_{j_2}$.
\item The group $H_{j}(S_1;\mathbb{Z})$ is trivial for $0<j<2n-2j_1-1$, $j \neq n-j_1-1$.
\item The group $H_{j}(S_2;\mathbb{Z})$ is trivial for $0<j<2n-j_1-j_2-1$, $j \neq n-j_1-1,n-j_2-1$. 
\end{enumerate}
Let ${\rm rank}_{\mathbb{Z}} G_{2j_1+1-n}>0$ and ${\rm rank}_{\mathbb{Z}} G_{j_1+j_2+1-n}>0$. Then we can construct a map $f^{\prime}$ satisfying the condition.
\end{enumerate}
\end{Thm}

We introduce an important class of manifolds. They are also used in \cite{kitazawa5}.
Let $k_1,k_2>0$ be integers.
For a closed and connected orientable $k_1$-dimensional manifold $S$, we can obtain a new manifold by gluing the two manifolds.
\begin{enumerate}
\item The product of  a compact manifold obtained by removing the interior of a smoothly embedded standard closed $k$-dimensional disc $D$ and $S^{k_2-1}=\partial D^{k_2}$
\item $D^{k_2} \times \partial D$.
\end{enumerate}
We glue them together by a product of  a diffeomorphisms on $\partial D^{k_2}$ and $\partial D$. We denote a manifold obtained by this procedure by $S_{(k_1,k_2)}$.
\begin{Lem}
\label{lem:1}
Let $a>0$ be an integer. If for a {\rm (}$k_1=2a+1${\rm )}-dimensional manifold $S$, $H_{j}(S;\mathbb{Z})$ is trivial for $0<j<a$ and $a<j<2a+1$, then $H_{j}(S_{(k_1,k_2)};\mathbb{Z})$ is trivial for $0<j<a$, $a<j<a+k_2-1$ and $a+k_2-1<j<k_1$ and two groups $H_{a}(S_{(k_1,k_2)};\mathbb{Z})$ and $H_{a+k_2-1}(S_{(k_1,k_2)};\mathbb{Z})$ are isomorphic to $H_{a}(S;\mathbb{Z})$
\end{Lem}
\begin{proof}[Proof of Theorem \ref{thm:4}]
Each statement is shown by taking a bouquet of the following closed submanifolds in an open disc in $W_f-q_f(S(f))$ and ${\rm rank}_{\mathbb{Z}} G_n-1$ points and applying Proposition \ref{prop:2}. For the bouquet, we take the following manifolds.
\begin{enumerate}
\item $S_1$.
\item $S_2$.
\item ${\rm rank}_{\mathbb{Z}} G_j$ copies of the ($n-j$)-dimensional standard sphere $S^{n-j}$ ($1 \leq j \leq n-1$ with $j \neq n-\dim S_1,n- \dim S_2$).
\item ${\rm rank}_{\mathbb{Z}} G_j-1$ copies of the ($n-j$)-dimensional standard sphere $S^{n-j}$ ($j = n-\dim S_1,n- \dim S_2$).
\end{enumerate}
We perform trivial M-bubbling operations whose generating polyhedra are the bouquet or the points one after another.
\end{proof}
\begin{Ex}
Let $n \geq 5$ and $(j_1,j_2)=(2,3)$ in the situation of Theorem \ref{thm:4}. We
 consider a case where the condition of the second statement holds: let $S_1$ be a $3$-dimensional closed, connected and orientable manifold and a Lens
 space whose 1st homology group is isomorphic to $G$ and let $S_2$ be a manifold which is, for a Lens space $S$ whose 1st homology group is isomorphic to $G_1$, diffeomorphic to $S_{(3,1)}$. Thus we may apply Theorem \ref{thm:4}.
\end{Ex}

\subsection{Several general statements}
We can show a general statement under the assumption that for several $G_j$ the torsion subgroups are non-trivial.
For an ordered set $A$, we denote the maximum and the minimum by $\max A$ and $\min A$, respectively.
\begin{Thm}
\label{thm:5}
Let $k>0$ be an integer and $\{A_j\}_{j=1}^{k}$ be a family of finite and non-empty set of integers satisfying the following.
\begin{enumerate}
\item Integers in $A_j$ are represented as powers of prime numbers.
\item $A_{j_1}$ and $A_{j_2}$ are disjoint for $j_1 \neq j_2$.
\end{enumerate}

In Problem in the introduction, assume that we can construct a map $f^{\prime}$ satisfying the condition.
 Let $I_{G,A_i}$ be the set of all numbers $j$ such that there exists an integer $p \in A_i$ and that ${\mathbb{Z}}_p$ is regarded as a summand of $G_j$ in Fact \ref{fact:3}. We assume that the following hold.
\begin{enumerate}
\item $I_{G,A_i}$ is non-empty for all $i$.
\item For any integers $1 \leq j_1 \leq j_2 \leq k$, the inequality $\max I_{G,A_{j_1}}<\min I_{G,A_{j_2}}$ holds.
\end{enumerate}
Then for each integer $1 \leq j \leq k$, for some integer $2\min I_{G,A_{j}}-n+1 \leq j^{\prime} \leq 2\max I_{G,A_{j}}-n+1$, ${\rm rank}_{\mathbb{Z}} G_{j^{\prime}}$ is positive. Moreover, at
 least $k>0$ of  $G_{j}$ are of ${\rm rank}_{\mathbb{Z}}G_j>0$ 
\end{Thm}
\begin{proof}
From the situation, there exists a bubbling operation in the family of  the bubbling operations to obtain $f^{\prime}$ from $f$ such that the generating polyhedron is represented as a bouquet
 of manifolds and that a summand of the torsion subgroup of a homology group of a manifold in the family of the manifolds is isomorphic to ${\mathbb{Z}}_p$ for $p \in A_j$ for each integer $1 \leq j \leq k$. The dimension
 of the manifold is larger than $n-\max I_{G,A_{j}}+(n-\max I_{G,A_{j}}-1)$ and smaller
 than $n-\min I_{G,A_{j}}+(n-\min I_{G,A_{j}}-1)$ by the assumption and by virtue of Poincar\'e duality theorem together with the universal coefficient theorem. Proposition \ref{prop:2} gives the
 desired result: $j^{\prime}$ in the statement satisfies the relation $n-(n-\min I_{G,A_{j}}+(n-\min I_{G,A_{j}}-1)) \leq j^{\prime} \leq n-(n-\max I_{G,A_{j}}+(n-\max I_{G,A_{j}}-1))$. The last statement follows from the latter two listed conditions. 
\end{proof}
\begin{Thm}
\label{thm:6}
In Problem in the introduction, assume that we can construct a map $f^{\prime}$ satisfying the conditions we will pose.
Let $k>0$ be an integer. Let $\{A_j\}_{j=1}^{k}$ be a family of finite and non-empty set of integers $1 \leq j^{\prime} \leq n-1$ and $\{H_j\}_{j=1}^{k}$ be a family of finite commutative groups satisfying the following.
\begin{enumerate}
\item For distinct integers $1 \leq j_1 \leq j_2 \leq k$, $H_{j_1}$ and $H_{j_2}$ are not isomorphic. 
\item For distinct integers $1 \leq j_1 \leq j_2 \leq k$, $A_{j_1}$ and $A_{j_2}$ are distinct. 
\item $A_{j}$ is the set of all integers $j^{\prime}$ such that $H_j$ is isomorphic to a summand of $G_{j^{\prime}}$ as in Fact \ref{fact:3}. Furthermore, if $G_{j^{\prime}}$ is represented as the direct sum of $H_j$ and a group ${H^{\prime}}_{j^{\prime}}$, for the family of all summands of $TH_j$ and that of $T{H^{\prime}}_{j^{\prime}}$ obtained as in Fact \ref{fact:3}, no pair of a summand in the former family and one in the latter family is isomorphic.  
\item For each $j^{\prime} \in A_{j}$, there exists just one subgroup isomorphic to $H_j$ in $G_{j^{\prime}}$ and for each integer $j^{\prime}$ not in $A_{j}$, there exists no group isomorphic to a summand of $H_j$ as a summand of $G_{j^{\prime}}$ as in Fact \ref{fact:3}.
\item For each $A_{j}$, the cardinality or the number of all elements is odd.
\end{enumerate}

 Then for each integer $1 \leq j \leq k$, for some integer $j^{\prime} \in A_j$, ${\rm rank}_{\mathbb{Z}} G_{2j^{\prime}-n+1}$ is positive. Moreover, at
 least $k>0$ of  $G_{j}$ are of ${\rm rank}_{\mathbb{Z}}G_j>0$ 
%
\end{Thm}
\begin{proof}
From the situation, for each integer $1 \leq j \leq k$, there exists at least one bubbling operation in the family of  the bubbling operations to obtain $f^{\prime}$ from $f$ such that the generating polyhedron is represented as a bouquet
 of manifolds and that the torsion subgroup of a homology group of a manifold in the family of the manifolds can be regarded as a group having a summand isomorphic to $H_j$: if there exists no such manifold, then we can obtain such a manifold by considering "a connected sum" instead of "a bouquet": note that this idea is a key ingredient in the proof of Proposition \ref{prop:5}, which is omitted. The dimension
 of at least one of manifolds whose homology groups are like this must be $(n-j^{\prime})+(n-j^{\prime}+1)$ for some $j^{\prime} \in A_j$ by the assumption and by virtue of Poincar\'e duality theorem together with the universal coefficient theorem: the most important ingredient is the fifth condition of the numbered conditions with the definition of $A_j$.
\end{proof}
\begin{Ex}
In the situation of Problem, let $n$ be sufficiently large and a fold map $f$ be given. Let $S_1$ be a $5$-dimensional closed, simply-connected and orientable manifold whose second homology group $H_2(S_1;\mathbb{Z})$ has a non-trivial torsion subgroup $TH_2(S_1;\mathbb{Z})$, appearing in \cite{barden}, and let $S_2$ be a manifold which is, for a Lens space $S$ whose 1st homology group is isomorphic to a group $G$, diffeomorphic to $S_{(3,2)}$. We also assume that for the family of all summands of the group $TH_2(S_1;\mathbb{Z})$ and that of the group $G$, no pair of a summand in the former family and one in the latter family is isomorphic. By performing a normal bubbling operation whose generating manifold is diffeomorphic to $S_1$ to $f$ and after the operation, performing one whose generating manifold is diffeomorphic to $S_2$, we have a map $f^{\prime}$ satisfying the assumption of Theorem \ref{thm:6} with $k=1$, $A_1$=\{n-3\} and $H_2(S_1;\mathbb{Z})=H_1$. Knowing more precise information on the situation such as groups $\{G_j\}$ is left to readers as an exercise.
\end{Ex}

\subsection{A case where there exist just three $G_j$ whose torsion-groups are non-trivial satisfying additional appropriate conditions in Problem in the introduction}
\begin{Thm}
\label{thm:7}
In Problem in the introduction, let there exist just three groups $G_{j_1}$, $G_{j_2}$ and $G_{j_3}$ {\rm (}$1 \leq j_1<j_2<j_3 \leq n-1${\rm )} in the family of the groups whose
 torsion subgroups are non-trivial and suppose that at least one of the following holds.
\begin{enumerate}
\item The two torsion subgroups $T_{G_{j_1}} \subset G_{j_1}$ and $T_{G_{j_3}} \subset G_{j_3}$ are not isomorphic.
\item $2j_2 \neq j_1+j_3$.
\end{enumerate}
 Assume also that we can construct a map $f^{\prime}$ satisfying the condition of Problem. Then there must be at least $2$ numbers $1 \leq j \leq n-1$ such that ${\rm rank}_{\mathbb{Z}} G_j>0$ hold.

\end{Thm}
\begin{proof}
Let the number of $j$ such that ${\rm rank}_{\mathbb{Z}} G_j>0$ hold be $0$ or $1$. From Proposition \ref{prop:2}, it must be $1$. Moreover, we can take the generating polyhedron as a closed and orientable manifold $S$ and for the homology group of the generating manifold, there exist just three numbers ${j_1}^{\prime}=j_1-(n-\dim S)$, ${j_2}^{\prime}=j_2-(n-\dim S)$
 and ${j_3}^{\prime}=j_3-(n-\dim S)$ such that the torsion subgroup $T H_{{j_{j^{\prime}}}^{\prime}}(S;\mathbb{Z})$ is isomorphic to $T_{G_{j_{j^{\prime}}}}$ ($j^{\prime}=1,2,3$).

 If the first condition of the numbered two conditions holds, then the Poincar\'e duality theorem together with the universal coefficient theorem implies that this is a contradiction. If the second condition holds and the first condition does not hold, then the assumption on the torsion subgroups of the three homology groups also contradicts the
 Poincar\'e duality theorem together with the universal coefficient theorem. 

We give a more precise explanation on the latter case. The relation ${j_1}^{\prime}=\dim S-{j_3}^{\prime}-1$ must hold by virtue of the Poincar\'e duality theorem together with the universal coefficient theorem and as a result the relation $j_1-(n-\dim S)=\dim S-(j_3-(n-\dim S))-1$ holds and we also have the relation $\dim S=2n-j_1-j_3-1$. From the Poincar\'e duality theorem, we also have the relation $2{j_2}^{\prime}+1=2(j_2-(n-\dim S))+1=\dim S$ and as a result the relation $\dim S=2n-2j_2-1$. This contradicts the second condition.

This completes the proof.
\end{proof}
I
\section{Remarks on manifolds}
The following gives important clues in knowing homology groups of manifolds admitting explicit fold maps. 
\begin{Fact}[\cite{kitazawa}, \cite{kitazawa2}, \cite{kitazawa3}, \cite{saeki2}, \cite{saekisuzuoka} etc.]
\label{fact:4}
For a fold map on a closed and connected manifold of dimension $m$ into an $n$-dimensional manifold with no boundary, let the relation $k=m-n>1$ hold. Then the quotient map onto the Reeb space induces isomorphisms of homology and homotopy groups of degree $l<k$ of the source manifold and those of the Reeb spaces if the following hold.
\begin{enumerate}
\item The map obtained by the restriction to the singular set onto the Reeb is injective {\rm (}a fold map is said to be {\rm simple} in this case: for the definition and fundamental and advanced properties, see \cite{saeki} and \cite{sakuma} for example{\rm )}.
\item Inverse images of regular values are disjoint unions of almost-spheres.
\item Indices of singular points are $0$ or $1$.
\end{enumerate}
If the map is special generic, then the quotient map onto the Reeb space induces isomorphisms of homology and homotopy groups of degree $l \leq k$ of the source manifold and those of the Reeb space.  
\end{Fact}
In short, if the assumption holds, the Reeb space inherits considerable information of the source manifold of the given fold map.

If we apply S-bubbling operations starting from special generic maps, then we can obtain maps satisfying the assumption of the fact. Maps presented in Example \ref{ex:1}, \ref{fact:2} etc. are simplest examples.

\end{document}